\documentclass[letterpaper,leqno,11pt,oneside]{amsart}

\usepackage{amsmath,amsthm,amsfonts,amssymb}
\usepackage{enumitem}
\usepackage[usenames]{color}
\usepackage{mathtools}
\usepackage[extension=pdf]{hyperref}
\usepackage[height=9.6in,width=5.95in]{geometry}
\usepackage{graphics,graphicx}
\usepackage[final,color,notref,notcite]{showkeys}
\usepackage{xr}

\usepackage[backend=bibtex,style=alphabetic,natbib=true,maxnames=99,isbn=false,doi=false,url=false,firstinits=true,hyperref=auto,arxiv=abs]{biblatex}
\DeclareFieldFormat[article,inbook,incollection,inproceedings,patent,thesis,unpublished]{title}{#1\isdot}
\renewbibmacro{in:}{%
  \ifentrytype{article}{}{%
  \printtext{\bibstring{in}\intitlepunct}}}
\defbibheading{apa}[\refname]{\section*{#1}}

\definecolor{darkblue}{rgb}{0.13,0.13,0.39}
\hypersetup{colorlinks=true,urlcolor=darkblue,citecolor=darkblue,linkcolor=darkblue,pdftitle={Supremum
  of Airy2 minus a parabola on a half line}}

\mathtoolsset{showonlyrefs,showmanualtags}

\newtheorem{thm}{Theorem}
\newtheorem{lem}{Lemma}[section]
\newtheorem{prop}[lem]{Proposition}

\theoremstyle{definition}
\newtheorem{rem}[lem]{Remark}
\newtheorem*{rem*}{Remark}

\newcounter{assum}

\newcommand{\I}{{\rm i}}

\newcommand{\Bt}{{\mathcal{A}}_{2\to 1}}
\newcommand{\pp}{\mathbb{P}}

\newcommand{\ee}{\mathbb{E}}
\newcommand{\rr}{\mathbb{R}}

\newcommand{\zz}{\mathbb{Z}}
\newcommand{\aip}{\mathcal{A}_2}
\newcommand{\aipo}{\mathcal{A}_1}
\newcommand{\ch}{\mathcal{H}}

\newcommand{\p}{\partial}
\newcommand{\uno}[1]{\mathbf{1}_{#1}}

\newcommand{\vs}{\vspace{6pt}}
\newcommand{\wt}{\widetilde}

\newcommand{\K}{K_{\Ai}}
\newcommand{\m}{\overline{m}}

\DeclareMathOperator{\Ai}{Ai}
\DeclareMathOperator{\tr}{tr}

\newcommand{\gref}[1]{\ref*{g-#1} of \cite{cqr}}
\newcommand{\geqref}[1]{(\ref*{g-#1}) in \cite{cqr}}

\numberwithin{equation}{section}

\let\oldmarginpar\marginpar
\renewcommand\marginpar[1]{\-\oldmarginpar[\raggedleft\footnotesize #1]%
{\raggedright{\small\textsf{#1}}}}

\begin{document}

\title[Supremum of the Airy$_2$ process minus a parabola on a half line]{Supremum of the Airy$_2$ process minus a parabola on a half line}

\author{Jeremy Quastel}
\address[J.~Quastel]{
  Department of Mathematics\\
  University of Toronto\\
  40 St. George Street\\
  Toronto, Ontario\\
  Canada M5S 2E4} \email{quastel@math.toronto.edu}
\author{Daniel Remenik}
\address[D.~Remenik]{
  Department of Mathematics\\
  University of Toronto\\
  40 St. George Street\\
  Toronto, Ontario\\
  Canada M5S 2E4 \newline \indent\textup{and}\indent
  Departamento de Ingenier\'ia Matem\'atica\\
  Universidad de Chile\\
  Av. Blanco Encala\-da 2120\\
  Santiago\\
  Chile} \email{dremenik@math.toronto.edu}

\maketitle

\begin{abstract}
Let $\aip(t)$ be the Airy$_2$ process.  We show that the random variable
\[\sup_{t\leq\alpha}\{\aip(t)-t^2\}+\min\{0,\alpha\}^2\]
has the same distribution as the one-point marginal of the Airy$_{2\to1}$ process at time
$\alpha$.  These marginals form a family of distributions crossing over from the GUE
Tracy-Widom distribution $F_{\rm GUE}(x)$ for the Gaussian Unitary Ensemble of random
matrices, to a rescaled version of the GOE Tracy-Widom distribution $F_{\rm GOE}( 4^{1/3}
x)$ for the Gaussian Orthogonal Ensemble. Furthermore, we show that for every $\alpha$ the
distribution has the same right tail decay $e^{-\frac43 x^{3/2} }$.
\end{abstract}

\section{Introduction}
The Airy processes are a collection of stochastic processes which are expected to govern
the spatial fluctuations of random growth models in the one dimensional KPZ universality
class for wide classes of initial data.  They are defined through their finite dimensional
distributions, which are given by Fredholm determinants.  The three basic processes are
Airy$_2$ \cite{prahoferSpohn}, corresponding to curved, or droplet initial data; Airy$_1$
\cite{sasamoto,borFerPrahSasam,bfp}, corresponding to flat initial data; and Airy$_{\rm stat}$
\cite{baikFerrariPeche}, corresponding to equilibrium initial data.

The KPZ class is identified at the roughest level by the unusual $t^{1/3}$ scale of
fluctuations. It is expected to contain a large class of random growth processes,
including the Kardar-Parisi-Zhang equation itself, as well as randomly stirred one
dimensional fluids, polymer chains directed in one dimension and fluctuating transversally
in the other due to a random potential (with applications to domain interfaces in
disordered crystals), driven lattice gas models, reaction-diffusion models in
two-dimensional random media (including biological models such as bacterial colonies),
randomly forced Hamilton-Jacobi equations, etc.  A combination of non-rigorous methods
(renormalization, mode-coupling, replicas) and mathematical breakthroughs on a few special
models has led to very precise predictions of universal scaling exponents and exact
statistical distributions describing the long time properties. These predictions have been
repeatedly confirmed through Monte-Carlo simulation as well as experiments; in particular,
recent spectacular experiments on turbulent liquid crystals by Takeuchi and Sano
\cite{takeuchiSano1,takeuchiSano2} have been able to even confirm some of the predicted
fluctuation statistics.

The conjectural picture that has developed is that the universality class is divided into
subuniversality classes which depend on the initial data class, but not on other details
of the particular models.  Because of their self-similarity properties, the three basic
initial data are, at the level of continuum partition functions (taking logarithms gives
free energies or height functions): Dirac $\delta_0$, corresponding to curved, or droplet
type initial data; $0$, corresponding to growth off a flat substrate; and $e^{B(x)}$ where
$B(x)$ is a two sided Brownian motion, corresponding to growth in equilibrium.  Of course,
in discrete models of various types one is dealing with discrete approximations of such
initial data.  There are also three additional non-homogeneous subuniversality classes
corresponding to starting with one of the basic three on one side of the origin, and
another on the other side.  The spatial fluctuations in these six basic classes of initial
data are supposed to be given asymptotically by the six known Airy processes: the three
basic Airy processes, Airy$_2$, Airy$_1$ and Airy$_{\rm stat}$, and the crossover Airy
processes Airy$_{2\to 1}$ \cite{bfs}, Airy$_{2\to {\rm BM}}$
\cite{imamSasam1,corwinFerrariPeche} and Airy$_{1\to {\rm BM}}$ \cite{bfsTwoSpeed}.

However, since all initial data are superpositions of Dirac masses, there is a sense is
which the Airy$_2$ process is the most basic.  Although the various microscopic models are
not linear in the initial data, this is the case for the stochastic heat equation, whose
logarithm is the solution of the KPZ equation.  And for other models, the linearity should
hold asymptotically.  In the limit, the logarithm of the superpositions of exponentials of
Airy$_2$ processes becomes a variational problem.

The conclusion is a conjecture that the one-point marginals of the other Airy processes
should be obtained through certain variational problems involving the Airy$_2$ process.
The first example of this was the celebrated result of \citet{johansson} (see also
\cite{cqr}) that the supremum of the Airy$_2$ process minus a parabola has the same
distribution as a rescaled version of the one-dimensional marginal of the Airy$_1$
process, i.e., the GOE Tracy-Widom distribution:
\begin{equation}
\pp\!\left(\sup_{x\in\rr}\big\{\aip(x)-x^2\big\}\leq m\right)=\pp\big(\aipo(0)\leq2^{-1/3}m\big)=F_{\rm GOE}(4^{1/3}m).\label{eq:GOE}
\end{equation}
The general conjecture is based on heuristics which we describe next in the context of the
stochastic heat equation.

\subsection{Heuristics}
\label{sec:heuristics}

We will explain the heuristics first for the case of the Airy$_{2\to1}$ process. Let
$z(t,x)$ denote the solution of the one-dimensional stochastic heat equation
 \begin{equation}
 \partial_t z = \tfrac12\partial_x^2 z - z \xi
 \end{equation}
 where $\xi(t,x)$ is space-time white noise.  The solution at position $x$ and time $t$
 starting from a Dirac mass at $y$ at time $0$ can be written as 
\begin{equation}\label{eq:zed}
  z(0,y; t,x) = \tfrac{1}{\sqrt{2\pi t}} e^{ -\tfrac{(x-y)^2}{2t} -\tfrac{t}{24}+ 2^{-1/3}t^{1/3}A_t(2^{-1/3}t^{-2/3}(x-y))},
\end{equation}
where $A_t$ is conjectured to converge to the Airy${}_2$ process, $A_t( x) \to \aip(x)$
(see Conjecture 1.5 in \cite{acq} for a precise statement, and \cite{prolhacSpohn} for a non-rigorous derivation). Starting from the step initial data $z(0,x) =\uno{x>0}$
the prediction is
\begin{equation}\label{eq:1to2approx}
  -\log z(t,x) \approx \tfrac1{2t}x^2\uno{x<0}+\tfrac{1}{24}t+\log(\sqrt{2\pi t}) - 2^{-1/3}t^{1/3} \Bt(2^{-1/3}t^{-2/3} x).
\end{equation}
On the other hand, by linearity we have for each fixed $x$, in distribution,
\begin{equation}\label{eq:zed2}
  z( t,x) =\int_0^\infty dy\,z(0,y; t,x)= \int_0^\infty dy\,\tfrac{1}{\sqrt{2\pi t}}
  e^{ -\tfrac{(x-y)^2}{2t} -\tfrac{t}{24}+ 2^{-1/3}t^{1/3}A_t(2^{-1/3}t^{-2/3} (x-y)) }.
\end{equation}
Note however that as written the equality can only hold in distribution for each $t$ and
$x$.  If one wants a stronger statement, for fixed $t$ but multiple $x$, one has to
replace $A_t(2^{-1/3}t^{-2/3} (x-y))$ in \eqref{eq:zed} by a two parameter process
$\tilde{A}_t(2^{-1/3}t^{-2/3} x, 2^{-1/3}t^{-2/3}y)$, keeping track of the statistical
dependence on the initial $y$.  For fixed $y$, it is distributionally, as a process in
$x$, equal to $A_t(2^{-1/3}t^{-2/3} (x-y))$. And, by symmetry, the same is true for fixed
$x$, as a process in $y$.  However, they are not equal in distribution in the sense of two
parameter processes in both $x$ and $y$. The limit of $\tilde{A}_t(x,y)$ is unknown at
this time, so one is stuck at the level of one-dimensional marginals.

Calling $\tilde{x}=2^{-1/3}t^{-2/3} x$ and
$\tilde{y}=2^{-1/3}t^{-2/3} y$ we can rewrite the exponent in \eqref{eq:zed2} as
\[2^{-1/3}t^{1/3}\big[A_t (\tilde{x}-\tilde{y}) -(\tilde{x}-\tilde{y})^2\big]-\tfrac1{24}t\] so that for large $t$ the fluctuation field
$2^{1/3}t^{-1/3} \big[\log z(t,x)+\frac1{24}t+\log(\sqrt{2\pi t})\big]$ is well approximated by
\begin{equation}
\sup_{\tilde{y}\ge 0} \big( \aip (\tilde{x}-\tilde{y})-(\tilde{x}-\tilde{y})^2\big).
\end{equation}
Comparing with \eqref{eq:1to2approx} we deduce that the processes $\sup_{y\ge 0} \big(
\aip(x-y)-(x-y)^2\big)$ and $\Bt(x)-x^2\uno{x<0}$ should have the same one-dimensional
distribution or, equivalently, that
\begin{equation}
  \Bt(x)-x^2\uno{x<0}\quad{\buildrel {\rm (d)} \over =}\quad\sup_{y\leq x} \big\{\aip(y)-y^2\big\}
\label{eq:conj}
\end{equation}
for each fixed $x\in\rr$.

The same argument works for the other two crossover cases. If we let
$z(0,x)=e^{B(x)}\uno{x\geq0}$, where $B(x)$ is a standard Brownian motion, then
\eqref{eq:1to2approx} and \eqref{eq:zed2} are replaced respectively by
\begin{gather}
  -\log z(t,x) \approx \tfrac1{2t}x^2\uno{x<0}+\tfrac{1}{24}t+\log(\sqrt{2\pi t}) -
  2^{-1/3}t^{1/3}\mathcal{A}_{2\to {\rm BM}}(2^{-1/3}t^{-2/3} x)\\
  \shortintertext{and}
  z( t,x)=\int_0^\infty dy\,z(0,y; t,x)= \int_0^\infty dy\,\tfrac{1}{\sqrt{2\pi t}}
  e^{ -\tfrac{(x-y)^2}{2t} -\tfrac{t}{24}+B(y)+2^{-1/3}t^{1/3}A_t(2^{-1/3}t^{-2/3} (x-y)) },
\end{gather}
and now the same scaling argument allows to conjecture that
\[\mathcal{A}_{2\to {\rm BM}}(x)-x^2\uno{x<0}\quad{\buildrel {\rm (d)} \over =}\quad\sup_{y\le x} \big( \aip
(y)+\widetilde{B}(x-y)-y^2\big)\]
for each fixed $x\in\rr$, where now $\widetilde B(y)$ is a
Brownian motion with diffusion coefficient 2. An analogous argument with
$z(0,x)=\uno{x\leq 0}+e^{B(x)}\uno{x\geq0}$ translates into conjecturing that
\[\mathcal{A}_{1\to {\rm BM}}(x)\quad{\buildrel {\rm (d)} \over =}\quad\sup_{y\in\rr} \big( \aip
(y)+\widetilde{B}(x-y)\uno{y\leq x}-y^2\big)\]
for each fixed $x\in\rr$.

In this article we prove the conjecture \eqref{eq:conj} for $\Bt$, which connects the
three Airy processes with non-random initial data. To state the result precisely, we now
recall the exact definitions of the Airy$_2$, Airy$_1$, and Airy$_{2\to 1}$ processes,
together with some additional background.

\subsection{Statement of main results.}
The Airy${}_2$ process $\aip$, introduced by \citet{prahoferSpohn}, is a stationary
process on the real line whose one dimensional marginals are given by the Tracy-Widom
largest eigenvalue distribution for the Gaussian Unitary Ensemble (GUE) from random matrix
theory \cite{tracyWidom}. It is expected to govern the asymptotic spatial fluctuations in
a wide variety of random growth models on a one dimensional substrate with curved initial
conditions, and the point-to-point free energies of directed random polymers in $1+1$
dimensions (the KPZ universality class).  It also arises as the scaling limit of the top
eigenvalue in Dyson's Brownian motion \cite{dyson} for GUE (see \cite{andGuioZeit} for
more details). It is defined through its finite-dimensional distributions, which are given
by a determinantal formula: given $\xi_1,\dots,\xi_m\in\mathbb{R}$ and $t_1<\dots<t_m$ in
$\mathbb{R}$,
 \begin{equation}\label{eq:detform}
   \pp\!\left(\aip(t_1)\le\xi_1,\dots,\aip(t_m)\le\xi_m\right) = 
   \det\!\big(I-{\rm f}^{1/2}K_2^{\mathrm{ext}}{\rm f}^{1/2}\big)_{L^2(\{t_1,\dots,t_m\}\times\mathbb{R})},
 \end{equation}
 where $\pp$ denotes probability, we have counting measure on $\{t_1,\dots,t_m\}$ and
 Lebesgue measure on $\mathbb{R}$, ${\rm f}$ is defined on
 $\{t_1,\dots,t_m\}\times\mathbb{R}$ by
 \begin{equation}
   {\rm f}(t_j,x)=\uno{x\in(\xi_j,\infty)},\label{eq:deff}
 \end{equation}
 and the {\it extended Airy kernel}, \cite{FNH,macedo,prahoferSpohn} is defined by
 \[K_2^\mathrm{ext}(t,\xi;t',\xi')=
 \begin{cases}
   \int_0^\infty d\lambda\,e^{-\lambda(t-t')}\Ai(\xi+\lambda)\Ai(\xi'+\lambda), &\text{if $t\ge t'$}\\
   -\int_{-\infty}^0 d\lambda\,e^{-\lambda(t-t')}\Ai(\xi+\lambda)\Ai(\xi'+\lambda),
   &\text{if $t<t'$,}
 \end{cases}\] where $\Ai(\cdot)$ is the Airy function.

 The Airy$_1$ process, introduced by \citet{sasamoto}, is another stationary process,
 whose one-point distribution is now given by the Tracy-Widom largest eigenvalue
 distribution for the Gaussian Orthogonal Ensemble (GOE) from random matrix theory
 \cite{tracyWidom2}.  It is defined through its finite-dimensional distributions,
\begin{equation}\label{eq:detAiry1}
\pp\!\left(\aipo(t_1)\le \xi_1,\dots,\aipo(t_n)\le \xi_n\right) = 
\det\!\big(I-\mathrm{f}^{1/2}K^{\mathrm{ext}}_1\mathrm{f}^{1/2}\big)_{L^2(\{t_1,\dots,t_n\}\times\mathbb{R})},
\end{equation}
with ${\rm f}$ as in \eqref{eq:deff} and
\begin{multline}\label{eq:fExtAiry1}
K^{\rm ext}_1(t,\xi;t',\xi')=-\frac{1}{\sqrt{4\pi
    (t'-t)}}\exp\!\left(-\frac{(\xi'-\xi)^2}{4 (t'-t)}\right)\uno{t'>t}\\
+\Ai(\xi+\xi'+(t'-t)^2) \exp\!\left((t'-t)(\xi+\xi')+\frac23(t'-t)^3\right).
\end{multline}
It  is expected to govern the asymptotic spatial fluctuations in  random growth models  with flat initial
conditions, and the point-to-line free energies of directed random polymers.

The Airy$_{2\to 1}$ process $\Bt$, introduced by \citet{bfs}, is given by
\begin{equation}\label{eqTransAiryProcess}
\pp\!\left(\Bt(t_1)\le\xi_1,\dots,\Bt(t_m)\le\xi_m\right)=
   \det\!\big(I-{\rm f}^{1/2}K_\infty{\rm f}^{1/2}\big)_{L^2(\{t_1,\dots,t_m\}\times\mathbb{R})},
\end{equation}
with ${\rm f}$ as in \eqref{eq:deff} and
\begin{multline}\label{eqKCompleteInfinity}
K_\infty(s,x;t,y)=-\frac{1}{\sqrt{4\pi(t-s)}}\exp\!\left(-\frac{(\tilde y-\tilde x)^2}{4(t-s)}\right)\uno{t>s}\\
+\frac{1}{(2\pi \I)^2} \int_{\gamma_+}dw \int_{\gamma_-}dz\,
\frac{e^{w^3/3+t w^2-\tilde y w}}{e^{z^3/3+s z^2-\tilde x z}} \frac{2w}{(z-w)(z+w)},
\end{multline}
where $\tilde x=x-s^2\uno{s\leq0}$, $\tilde y=y-t^2\uno{t\leq 0}$ and
the paths $\gamma_+,\gamma_-$ satisfy $-\gamma_+\subseteq\gamma_-$ with
\mbox{$\gamma_+:e^{\I\phi_+}\infty\to e^{-\I\phi_+}\infty$},
\mbox{$\gamma_-:e^{-\I\phi_-}\infty\to e^{\I\phi_-}\infty$} for some $\phi_+\in
(\pi/3,\pi/2)$, $\phi_-\in (\pi/2,\pi-\phi_+)$. The Airy$_{2\to 1}$ process crosses over
between the Airy$_2$ and the Airy$_1$ processes in the sense that $\Bt(t+\tau)$ converges
to $2^{1/3}\aipo(2^{-2/3}\tau)$ as $t\to\infty$ and $\aip(\tau)$ when $t\to-\infty$. It is
expected to govern the asymptotic spatial fluctuations in random growth models when the
initial conditions are {\it half flat}. In particular, it is shown in \cite{bfs} that it
governs the asymptotic fluctuations for the totally asymmetric (to the left) simple
exclusion process starting with particles only at the even positive integers.
 
Define the \emph{crossover distributions} $G^{2\to1}_\alpha$, for $\alpha\in\rr$, as
follows:
\begin{equation}
G^{2\to1}_\alpha(m)=\det\!\big(I-P_mK_\alpha P_m\big)\label{eq:Ga}
\end{equation}
where $P_m$ denotes the projection onto the interval $[m,\infty)$,
$K_\alpha=K^1_\alpha+K^2_\alpha$ and the kernels $K^1_\alpha$ and $K^2_\lambda$ are given
by
\[K_\alpha^1(x,y)=\int_0^{\infty} d\lambda\,e^{2\alpha\lambda}\Ai(x-\lambda+\max\{0,\alpha\}^2)\Ai(y+\lambda+\max\{0,\alpha\}^2)\]
and
\[K_\alpha^2(x,y)= \int_0^\infty
d\lambda\Ai(x+\lambda+\max\{0,\alpha\}^2)\Ai(y+\lambda+\max\{0,\alpha\}^2).\] Here, and in
everything that follows, the determinant means the Fredholm determinant in the Hilbert
space $L^2(\rr)$. As noted in Appendix A of \cite{bfs}, the kernel $K_\infty$ can be
expressed in terms of Airy functions\footnote{This corresponds to a minor correction of
  the formula appearing in \cite{bfs}, where the exponential prefactor appears in front of
  $L_0$ instead of $L_1+L_2$.}:
\begin{equation}
  K_\infty(s,t;x,y)=L_0(s,x;t,y)+e^{2t^3/3-2s^3/3+t\tilde y-s\tilde x}[L_1+L_2](s,x;t,y),\label{eq:Kinfty2}
\end{equation}
where
\begin{align}
      L_0(s,x;t,y)&=-e^{(s-t)\Delta}(\tilde x,\tilde
      y)=-\frac{1}{\sqrt{4\pi(t-s)}}e^{-(\tilde x-\tilde y)^2/4(t-s)},\\
      L_1(s,x;t,y)&=\int_0^\infty d\lambda\,e^{\lambda(s+t)}\Ai(\hat x-\lambda)\Ai(\hat
      y+\lambda),\\
      L_2(s,x,t,y)&=\int_0^\infty d\lambda\,e^{\lambda(t-s)}\Ai(\hat x+\lambda)\Ai(\hat
    y+\lambda)
\end{align}
with $\tilde x=x-s^2\uno{s\leq0}$, $\tilde y=y-t^2\uno{t\leq0}$, $\hat
x=x+s^2\uno{s\geq0}$ and $\hat y=y+t^2\uno{t\geq0}$.  Using this for $s=t=\alpha$ it is
straightforward to check that $K_\infty(t,\cdot;t,\cdot)$ is just a similarity transform
of the kernel $K_\alpha$, and therefore
\[G^{2\to1}_\alpha(m)=\pp\!\left(\Bt(\alpha)\leq m\right).\]

\begin{rem}
  The operator $P_mK_\alpha P_m$ appearing inside the determinant defining
  $G^{2\to1}_\alpha$ in \eqref{eq:Ga} is trace class (this follows from \eqref{eq:K1anorm}
  together with a similar bound for $K^2_\alpha$). This should be compared with the fact
  that the extended kernel given in \cite{bfs} for the higher dimensional joint
  distributions of $\Bt$ is not trace class (but, as shown in Appendix B of \cite{bfs},
  there is a conjugate kernel which is).
\end{rem}

The following result confirms the conjecture \eqref{eq:conj}:

\begin{thm}\label{thm:1to2}
  Fix $\alpha\in\rr$. For every $m\in\rr$,
  \[\pp\!\left(\sup_{t\leq\alpha}\big(\aip(t)-t^2\big)\leq m-\min\{0,\alpha\}^2\right)=G^{2\to1}_\alpha(m).\]
\end{thm}

We remark that the equality is easy to obtain in the limits $\alpha\to\infty$ and
$\alpha\to-\infty$.  Since the Airy$_2$ process is stationary one expects that, as
$\alpha\to-\infty$, $\sup_{t\leq\alpha}\big(\aip(t)-t^2\big)$ is attained at
$t\approx\alpha$, and thus
$\sup_{t\leq\alpha}\big(\aip(t)-t^2\big)+\min\{0,\alpha\}^2\approx\aip(\alpha)$, which has
distribution $F_{\rm GUE}(m)$.  For the right hand side, observe that for $\alpha<0$ the
kernel $K_\alpha^2$ equals the \emph{Airy kernel} $\K$, which is given explicitly by
\begin{equation}
\label{airykernel}
\K(x,y)=\int_0^{\infty}d\lambda\Ai(x+\lambda)\Ai(y+\lambda).
\end{equation}
In addition, one can check that, due to the exponential factor
$e^{2\alpha\lambda}$, $P_mK_\alpha^1P_m$ goes to 0 in trace norm as $\alpha\to-\infty$ (see
\eqref{eq:G21toK}). Consequently,
\begin{equation}\label{eq:toGUE}
\lim_{\alpha\to-\infty}G^{2\to1}_\alpha(m)=\det\!\big(I-P_m\K P_m\big)
\end{equation}
which is also $F_{\rm GUE}(m)$.
  
On the other hand, as $\alpha\to\infty$, the left hand side becomes
$\pp\!\left(\sup_{t\in\rr}\big(\aip(t)-t^2\big)\leq m\right)=F_{\rm
  GOE}(4^{1/3}m)$. For  the right hand side, it is not hard to check that $P_mK^2_\alpha P_m$ goes to 0 in trace norm as
$\alpha\to\infty$. In addition, using \eqref{eq:Kinfty2} and (A.6) of \cite{bfs} with
$\tau_1=\tau_2=\alpha$ we get for $\alpha>0$
\[K^1_\alpha(x,y)=2^{-1/3}\Ai(2^{-1/3}(x+y)) -\overline K^1_\alpha(x,y)\] with
$\overline K^1_\alpha(x,y)=\int_{-\infty}^0d\lambda\,e^{2\alpha\lambda}\Ai(x+\alpha^2+\lambda)\Ai(y+\alpha^2-\lambda)$,
and one can check that $P_m\overline K^1_\alpha P_m$ goes to 0 in trace norm as
$\alpha\to\infty$ (see the comment following \eqref{eq:G21toK}). The first term on the right hand side above
corresponds to the kernel $\wt B(x,y)=2^{-1/3}\Ai(2^{-1/3}(x+y))$, and we deduce after
changing variables $x\mapsto2^{-2/3}x$, $y\mapsto2^{-2/3}y$ in the resulting determinant
that
\begin{equation}
  \lim_{\alpha\to\infty}G^{2\to1}_\alpha(m)=\det\!\big(I-P_{4^{1/3}m}BP_{4^{1/3}m}\big),\label{eq:toGOE}
\end{equation}
where $B(x,y)=\frac12\Ai(\frac12(x+y))$, which is also equal to  $F_\mathrm{GOE}(4^{1/3}m)$
\cite{ferrariSpohn}. 

The fact that $G^{2\to1}_\alpha$ crosses over between the GUE and GOE distributions is of
course a particular case of the crossover property of the Airy$_{2\to1}$ process. Note
that the scaling by $4^{1/3}$ in the GOE end of this interpolation implies that both ends
satisfy the same asymptotics $\log\big(1-G^{2\to1}_{\pm\infty}(m)\big)\sim-\frac43m^{3/2}$
as $m\to\infty$\footnote{This is due to the known asymptotics $\log\!\big(1-F_{\rm
    GOE}(m)\big)\sim-\frac23m^{3/2}$ and $\log\!\big(1-F_{\rm GUE}(m)\big)\sim-\frac43m^{3/2}$,
  which follow from the formulas for these distributions in terms of the Painlev\'e II
  function \cite{tracyWidom,tracyWidom2}.}. In fact, the same upper bound holds for all
$\alpha\in\rr$:
  
\begin{prop}\label{prop:tail} For every $\alpha\in \rr$, there is a $c>0$ such that,
  \begin{equation}
    1-G^{2\to1}_{\alpha}(m)\leq cm\hspace{0.1em}e^{-\frac43m^{3/2}+\alpha m}\qquad\text{as}~m\to\infty.
  \end{equation}
\end{prop}

The proof of Theorem \ref{thm:1to2} is based on a continuum statistics formula for the
Airy$_2$ process, developed in \cite{cqr}, which is well adapted to such variational
problems. Fix a function $g\in H^1([\ell,r])$ and introduce an operator
$\Theta^g_{[\ell,r]}$ which acts on $L^2(\rr)$ as follows:
$\Theta^g_{[\ell,r]}f(\cdot)=u(r,\cdot)$, where $u(r,\cdot)$ is the solution at time $r$
of the boundary value problem
\begin{equation}
  \begin{aligned}
    \p_tu+Hu&=0\quad\text{for }x<g(t), ~t\in (\ell,r)\\
    u(\alpha,x)&=f(x)\uno{x<g(\alpha)}\\
    u(t,x)&=0\quad\text{for }x\ge g(t)
  \end{aligned}\label{eq:bdval}
\end{equation} for the \emph{Airy Hamiltonian},
\[H=-\p_x^2+x.\]
The formula  reads
\begin{equation}
  \pp\!\left(\aip(t)\leq g(t)\text{ for }t\in[\ell,r]\right)
  =\det\!\left(I-K_{\Ai}+\Theta^g_{[\ell,r]}e^{(r-\ell)H}\K\right).\label{eq:firstprob}
\end{equation}
Choosing $g(t)=m-\min\{0,\alpha\}^2+t^2$ and $r=\alpha$,
\eqref{eq:firstprob} gives an explicit formula for the probability in Theorem
\ref{thm:1to2} in the limit $\ell\to-\infty$. The proof will consist on computing this
limit and showing that it coincides with $G^{2\to1}_\alpha(m)$.

Note that this strategy is considerably more difficult to implement for the Airy$_{\rm
  stat}$, Airy$_{2\to {\rm BM}}$ and Airy$_{1\to {\rm BM}}$ processes, because it involves
computing an expectation of the Fredholm determinant in \eqref{eq:firstprob} with respect
to Brownian motion paths. For example, in the stationary case, for which the one-point
distribution is given by $\pp\!\left(\mathcal{A}_{\rm stat}(t)\leq m\right)=F_0(m)$ with
$F_0$ the Baik-Rains distribution \cite{baikRainsF0}, we would need to check the formula
\begin{equation}\label{cc1'}
F_0(m) =\lim_{L\to \infty}\ee\!\left(\det\!\left(I-K_{\Ai}+\Theta^{B(\cdot)+(\cdot)^2 +m}_{[-L,L]}  e^{2LH}\K\right)\right),
\end{equation} 
where $\Theta^{B(\cdot)+(\cdot)^2 +m}_{[-L,L]}$ has an explicit kernel, which can be
derived from Theorem \gref{thm:thetaLgen}, and is given by
\begin{equation}\label{eq:ThetaL}
\Theta^{B(\cdot)+(\cdot)^2 +m}_{[-L,L]} (x,y)=e^{-L( x+y)+2L^3/3}\frac{e^{-(x-y)^2/8L}}{\sqrt{8\pi L}}
\widehat{\pp}_{\hat B(-L)=x-L^2\atop\hat B(L)=y-L^2}\!\left(\hat B(s)\leq B(s)+m\text{ on }[-L,L]\right),
\end{equation}
where $\hat B$ a Brownian bridge from $x-L^2$ at time $-L$ to $y-L^2$ at time $L$, $B$ is
an independent two sided Brownian motion with $B(0)=0$ (both with diffusion coefficient
$2$), and the expectation $\ee$ in \eqref{cc1'} is with respect to $B$ while the
probability $\widehat\pp$ in \eqref{eq:ThetaL} is with respect to $\hat B$.

\subsection{Connection with last passage percolation}

It is worth remarking that the general picture we have described holds essentially exactly
at the discrete level in the case of last passage percolation. Here one considers a
family $\big\{\hspace{-0.05em}w(i,j)\}_{i,j\in\zz^+}$ of independent identically distributed random
variables and lets $\Pi_n$ be the collection of up-right paths of length $n$, that is,
paths $\pi=(\pi_0,\dotsc,\pi_n)$ such that $\pi_i-\pi_{i-1}\in\{(1,0),(0,1)\}$. The
\emph{point-to-point last passage time} is defined, for $m,n\in\zz^+$, by
\[L^{\rm point}(m,n)=\max_{\pi\in\Pi_{m+n}:(0,0)\to(m,n)}\sum_{i=0}^{m+n}w(\pi(i)),\] where the
notation in the subscript in the maximum means all up-right paths connecting the origin to
$(m,n)$. Similarly, the \emph{point-to-line last passage time} is defined by
\[L^{\rm line}(n)=\max_{k=-n,\dotsc,n}L^{\rm point}(n-k,n+k).\]

Next one defines the process $t\mapsto H^{\rm point}_n(t)$ by linearly interpolating the
values given by scaling $L^{\rm point}(m,n)$ through the relation
\begin{equation}\label{eq:10010}
L^{\rm point}(n+y,n-y)=c_1n+c_2n^{1/3}H^{\rm point}_n(c_3n^{-2/3}y),
\end{equation}
where the constants $c_i$ depend only on the distribution of the $w(i,j)$. In the special
case where the $w(i,j)$ have a geometric distribution, \citet{johansson} showed that
\begin{equation}\label{eq:johFCLT}
  H^{\rm point}_n(t) \to \aip(t)-t^2
\end{equation}
in distribution, in the topology of uniform convergence on compact sets, where $\aip$ is
the Airy${}_2$ process. One can also define a rescaled version $H^{\rm line}_n$ of
$L^{\rm line}(n)$, obtaining the relation
\[H^{\rm line}_n=\sup_{t\in\rr}\big\{H^{\rm point}(t)\big\}\] (here we are setting $H^{\rm
  point}_n(t)=0$ for $|t|>c_3n^{1/3}$). It is known \cite{baikRains} that $H^{\rm line}_n$
converges in distribution to a GOE Tracy-Widom random variable, and hence
\eqref{eq:johFCLT} allows to take $n\to\infty$ in the last equality to recover
\eqref{eq:GOE} (this was Johansson's original proof of \eqref{eq:GOE}, the proof in
\cite{cqr} is based on \eqref{eq:bdval}).

In principle, this idea can be extended to the obtain variational formulas for the other
Airy processes. For example, one could attempt to replace point-to-line last passage times
by point-to-half-line last passage times to recover \eqref{eq:conj}. Unfortunately, the
connection between Airy$_{2\to1}$ (and the other Airy processes) and last passage
percolation is made through translating the corresponding results for the totally
asymmetric exclusion process, and by doing this the boundary conditions end up away from
the line $\{(n-k,n+k),\,k=-n,\dotsc,n\}$, so \eqref{eq:johFCLT} is not directly
applicable. In work in progress \citet{corwinLiuWang}  obtain an
improved version of the slow decorrelation result proved in \cite{corwinFerrariPeche1},
which would lead to a general version of formulas for last passage times in
last passage percolation in terms of variational problems for the Airy$_2$ process. In
particular,  such a result would give a proof of the conjectures made in Section
\ref{sec:heuristics} and, as a consequence, would show that \eqref{cc1'}, and similar
formulas for the other Airy processes, hold.

\vs

\paragraph{\bf Acknowledgements}
JQ and DR were supported by the Natural Science and Engineering Research Council of
Canada, and DR was supported by a Fields-Ontario Postdoctoral Fellowship and by Fondecyt
Grant 1120309. Part of this work was done during the Fields Institute program ``Dynamics
and Transport in Disordered Systems" and the authors would like to thank the Fields
Institute for its hospitality.

\section{Derivation of the formula}\label{sec:deriv}


As in \cite{cqr,mqr} we will first give an expression for the distribution of
the supremum over a finite interval and then take a limit. Thus we choose an $L>-\alpha$, which
will later be taken to infinity, and work on the interval $[-L,\alpha]$. For notational
simplicity we will write
\[\overline m=m-\min\{0,\alpha\}^2.\]


We recall that the shifted Airy functions $\phi_\lambda(x)=\Ai(x-\lambda)$ are the
generalized eigenfunctions of the Airy Hamiltonian, as
$H\phi_\lambda=\lambda\phi_\lambda$, and the Airy kernel $K_{\Ai}$ is the projection of
$H$ onto its negative generalized eigenspace (see Remark \gref{airyrem}). Therefore
$e^{(\alpha+L)H}\K$ has integral kernel
\begin{equation}
  \label{eq:eLHK}
  e^{(\alpha+L)H}\K(x,y)=\int_{-\infty}^0d\lambda\,e^{(\alpha+L)\lambda}\Ai(x-\lambda)\Ai(y-\lambda).
\end{equation}
We also deduce that $e^{(\alpha+L)H}\K=\K e^{(\alpha+L)H}\K$, so we may use the cyclic property
of determinants to rewrite \eqref{eq:firstprob} as
\begin{equation}
  \pp\!\left(\aip(t)\leq g(t)\text{ for }t\in[-L,\alpha]\right)
  =\det\!\left(I-K_{\Ai}+e^{(\alpha+L)H}\K\Theta^g_{[-L,\alpha]}\K\right).
  \label{eq:aiL}
\end{equation}

We will apply the above for $g(t)=t^2+\m$, and for this choice of $g$ we will write
$\Theta_{\alpha,L}$ for $\Theta^g_{[-L,\alpha]}$. In this case the resulting kernel can be
computed explicitly, and a minor variation of \geqref{eq:thetaL} gives
\begin{equation}\label{eq:thetag}
  \Theta_{\alpha,L}(x,y)=\bar P_{\m+L^2}e^{-(\alpha+L)H}\bar P_{\m+\alpha^2}-\bar P_{\m+L^2}R_{\alpha,L}\bar P_{\m+\alpha^2},
\end{equation}
where $\bar P_m=I-P_m$ denotes the projection onto the interval $(-\infty,m]$ and $R_{\alpha,L}$ is given by
\begin{equation}
  \label{eq:R}
  R_{\alpha,L}(x,y)=\frac{1}{\sqrt{4\pi(\alpha+L)}}e^{-Lx-\alpha y+(\alpha^3+L^3)/3-\frac{(x-L^2+y-\alpha^2-2\m)^2}{4(\alpha+L)}}.
\end{equation}
Following \cite{cqr} we decompose $\Theta_{\alpha,L}$ as
\[\Theta_{\alpha,L}=e^{-(\alpha+L)H}\bar P_{\m+\alpha^2}-R_{\alpha,L}\bar P_{\m+\alpha^2}-\Omega_{\alpha,L},\]
where $\Omega_{\alpha,L}=P_{\m+L^2}(e^{-(\alpha+L)H}-R_{\alpha,L})\bar P_{\m+\alpha^2}$. Using this in
\eqref{eq:aiL} we can write
\begin{multline}
\pp\!\left(\aip(t)\leq t^2\text{ for }t\in[-L,\alpha]\right)\\
  =\det\!\left(I-\K P_{\m+\alpha^2}K_{\Ai}-
    e^{(\alpha+L)H}\K R_{\alpha,L}\bar P_{\m+\alpha^2}K_{\Ai}
    -e^{(\alpha+L)H}\K\Omega_{\alpha,L}K_{\Ai}\right).\label{eq:aiL2}
\end{multline}
We will show below that
\begin{equation}
  \label{eq:omega}
  \wt\Omega_{\alpha,L}:=e^{(\alpha+L)H}\K\Omega_{\alpha,L}\K\xrightarrow[L\to\infty]{}0
\end{equation}
in trace norm. Then since the mapping $A\mapsto\det(I+A)$ is continuous in the space of
trace class operators (see \eqref{eq:detCont} below), all that is left to do in order to
take $L\to\infty$ in \eqref{eq:aiL2} is to compute the limit of the operator
$e^{(\alpha+L)H}\K R_{\alpha,L}$

We will first proceed formally to identify the limit, and then verify it in Lemma
\ref{lem:goodBCH}. Since $\K$ is a projection and $H$ leaves $\K$ invariant, we will
pretend that $e^{(\alpha+L)H}$ and $\K$ commute, so we have to compute the limit of
$e^{(\alpha+L)H}R_{\alpha,L}$.  Define the \emph{reflection operator} $\varrho_m$ by
\[\varrho_mf(x)=f(2m-x).\]
Then the operator $R_{\alpha,L}$ defined in \eqref{eq:R} can be rewritten as
\begin{equation}
R_{\alpha,L}=e^{(\alpha^3+L^3)/3}e^{-L\xi}e^{(\alpha+L)\Delta}\varrho_{\m+(L^2+\alpha^2)/2}e^{-\alpha\xi}.\label{eq:Rgauss}
\end{equation}
Here $e^{r\xi}$ ($\xi$ stands for a generic variable) denotes the multiplication operator
$(e^{r\xi}f)(x)=e^{rx}f(x)$ and $\Delta$ is the Laplacian, so $e^{r\Delta}$ is the heat
kernel as in the introduction. The advantage of this form for $R_{\alpha,L}$ is that it
will allow us to use the Baker-Campbell-Hausdorff formula to compute formally
$e^{(\alpha+L)H}R_{\alpha,L}$.

We will use the following identities, where $[\cdot,\cdot]$ denotes commutator:
\[[H,\Delta]=[\xi,\Delta]=-2\nabla,\qquad[H,\nabla]=[\xi,\nabla]=-I,\qquad[H,\xi]=-2\nabla.\]
If $A$ and $B$ are two operators such that $[A,[A,B]]=c_1I$ and $[B,[A,B]]=c_2I$ for some
$c_1,c_2\in\rr$, then the Baker-Campbell-Hausdorff formula\footnote{The
  Baker-Campbell-Hausdorff formula can be found in most introductory books on Lie groups
  and algebras. A general version can be found in \cite{dynkinBCH}.  However, in this very simple context, it is more readily computed by
  hand.} reads
\begin{equation}
  \label{eq:bch}
  e^{A}e^{B}=e^{A+B+\frac12[A,B]+\frac1{12}[A,[A,B]]-\frac1{12}[B,[A,B]]}.
\end{equation}
In particular, if $[A,B]=cI$ then
\begin{equation}
  \label{eq:bch2}
  e^{A+B}=e^{A}e^{B}e^{-\frac12[A,B]}.
\end{equation}
Using \eqref{eq:bch} we have
\[e^{-L\xi}e^{(\alpha+L)\Delta}=e^{L^2(L+\alpha)/6}e^{(\alpha+L)\Delta+L(\alpha+L)\nabla-L\xi}.\]
Using \eqref{eq:bch} again we deduce that
\begin{align}
  e^{(\alpha+L)H}e^{-L\xi}e^{(\alpha+L)\Delta}&=e^{L^2(\alpha+L)/6}e^{(\alpha+L)H}e^{(\alpha+L)\Delta+L(\alpha+L)\nabla-L\xi}\\
  &=e^{\alpha^3/6-\alpha L^2/2-L^3/3}e^{\alpha\xi+(L^2-\alpha^2)\nabla}.
\end{align}
By \eqref{eq:bch2} we have $e^{\alpha\xi+(L^2-\alpha^2)\nabla}
=e^{\alpha(L^2-\alpha^2)/2}e^{\alpha\xi}e^{(L^2-\alpha^2)\nabla}$,
so the above identity gives
\[e^{(\alpha+L)H}e^{-L\xi}e^{(\alpha+L)\Delta}=e^{-(\alpha^3+L^3)/3}e^{\alpha\xi}e^{(L^2-\alpha^2)\nabla}.\]
Using this in \eqref{eq:Rgauss} we deduce that
\[e^{(\alpha+L)H}R_{\alpha,L}
=e^{\alpha\xi}e^{(L^2-\alpha^2)\nabla}\varrho_{\m+(L^2+\alpha^2)/2}e^{-\alpha\xi}.\]
Since $e^{r\nabla}$ is the shift operator $(e^{r\nabla}f)(x)=f(x+r)$, we have
$e^{r\nabla}\varrho_m=\varrho_{m-r/2}$, and we obtain
\[e^{(\alpha+L)H}R_{\alpha,L}=e^{\alpha\xi}\varrho_{\m+\alpha^2}e^{-\alpha\xi}.\]
The conclusion from the above is the following

\begin{lem}\label{lem:goodBCH}
\[e^{(\alpha+L)H}\K R_{\alpha,L}=\K e^{\alpha\xi}\varrho_{\m+\alpha^2}e^{-\alpha\xi}.\]
\end{lem}

We postpone the proof of this lemma until the end of this section. Putting this formula and \eqref{eq:omega} in
\eqref{eq:aiL2} and using Lemma \gref{lem:fredholm} gives
\begin{equation}\label{eq:limL}
  \begin{aligned}
    \pp\!\left(\sup_{t\leq\alpha}\big(\aip(t)-t^2\big)\leq\m\right)
    &=\lim_{L\to\infty}\pp\!\left(\aip(t)\leq t^2\text{ for }t\in[-L,\alpha]\right)\\
    &=\det\!\left(I-\K P_{\m+\alpha^2}K_{\Ai}-\K e^{\alpha\xi}\varrho_{\m+\alpha^2}
      e^{-\alpha\xi}\bar P_{\m+\alpha^2}\K\right).
  \end{aligned}
\end{equation}

Having obtained an expression for the probability we are interested in, all that remains
to show is that it coincides with our definition of $G^{2\to1}_\alpha$ \eqref{eq:Ga}. We
recall (as can be seen directly from its definition \eqref{airykernel}) that the Airy
kernel can be expressed as $\K=A\bar P_0A^*$, where $A$ is the \emph{Airy transform} which
acts on $f\in L^2(\rr)$ as
\[Af(x)=\int_{-\infty}^\infty dz\Ai(x-z)f(z).\]
Since $A^*=A^{-1}$ we have by the cyclic property of determinants that the right hand side of
\eqref{eq:limL} equals
\[\det\!\left(I-\bar P_0A^*P_{\m+\alpha^2}A\bar P_0-\bar P_0A^*e^{\alpha\xi}
     \varrho_{\m+\alpha^2}e^{-\alpha\xi}\bar P_{\m+\alpha^2}A\bar P_0\right).\]
Recalling that $\varrho_0f(x)=f(-x)$ we have similarly that the last determinant equals
\begin{multline} 
  \det\!\left(I-\varrho_0\bar P_0A^*P_{\m+\alpha^2}A\bar P_0\varrho_0-\varrho_0\bar P_0A^*
    e^{\alpha\xi}\varrho_{\m+\alpha^2}e^{-\alpha\xi}\bar P_{\m+\alpha^2}A\bar P_0\varrho_0\right)\\
  =\det\!\left(I-P_0\varrho_0 A^*P_{\m+\alpha^2}A\varrho_0 P_0-P_0\varrho_0 A^*
    e^{\alpha\xi}\varrho_{\m+\alpha^2}e^{-\alpha\xi}\bar P_{\m+\alpha^2}A\varrho_0 P_0\right).
\end{multline}
Shifting the variables in the last determinant by $-m$ we deduce that
\begin{equation}
  \label{eq:modLim}
  \pp\!\left(\sup_{t\geq\alpha}\big(\aip(t)-t^2\big)\leq\m\right)
  =\det\!\left(I-P_mE_1P_m-P_mE_2P_m\right),
\end{equation}
where
\begin{align}
  E_1(x,y)&=\int_{-\infty}^{\m+\alpha^2}d\lambda
  \Ai(x-m+2\m+2\alpha^2-\lambda)e^{-2(\lambda-\m-\alpha^2)\alpha}\Ai(y-m+\lambda)\\
  \shortintertext{and}
  E_2(x,y)&=\int_{\m+\alpha^2}^\infty d\lambda \Ai(x-m+\lambda)\Ai(y-m+\lambda).
\end{align}
Shifting $\lambda$ by $\m+\alpha^2$ in both integrals and changing $\lambda$ to $-\lambda$
shows that $E_1(x,y)=K^1_\alpha(y,x)$ and $E_2=K^2_\alpha$,
whence the equality in Theorem \ref{thm:1to2} follows since $E_1^*=K^1_\alpha$ and $E_2^*=K_\alpha^2$.

All we have left is to prove \eqref{eq:omega}.  We will denote by $\|\cdot\|_{\rm op}$,
$\|\cdot\|_1$ and $\|\cdot\|_2$ respectively the operator, trace class and Hilbert-Schmidt
norms of operators on $L^2(\rr)$ (see Section \gref{sec:aiL} for the definitions or
\cite{simon} for a complete treatment). Recall that
\begin{equation}\label{eq:norms}
  \|AB\|_1\leq\|A\|_2\|B\|_2\qquad\text{and}\qquad\|AB\|_2\leq\|A\|_2\|B\|_{\rm op}.
\end{equation}

\begin{proof}[Proof of \eqref{eq:omega}]
  Let $\varphi(x)=(1+x^2)^{1/2}$ and define the multiplication operator
  $Mf(x)=\varphi(x)f(x)$. Then by \eqref{eq:norms} we have that
    \[\|\wt\Omega_{\alpha,L}\|_1\leq\|e^{(\alpha+L)H}\K M^{-1}\|_2\,\|MP_{\m+L^2}(e^{-(\alpha+L)H}-R_{\alpha,L})\bar
    P_{\m+\alpha^2}\K\|_2\]
    Now $\|e^{(\alpha+L)H}\K M^{-1}\|_2=\|M^{-1}e^{(\alpha+L)H}\K\|_2$ by the symmetry of
    $e^{(\alpha+L)H}\K$, and then \geqref{eq:sndHS} gives $\|e^{(\alpha+L)H}\K
    M^{-1}\|_2\leq c(\alpha+L)^{-1/2}$. Then to finish the proof it will be enough to
    estimate $\|MP_{\m+L^2}e^{-(\alpha+L)H}\bar P_{\m+\alpha^2}\K\|_2$ and
    $\|MP_{\m+L^2}R_{\alpha,L}\bar P_{\m+\alpha^2}\K\|_2$.

  We start with the second norm. By \eqref{eq:R}, \eqref{eq:norms} and the fact that $\K$
  is a projection we have
  \begin{multline}
    \|MP_{\m+L^2}R_{\alpha,L}\bar P_{\m+\alpha^2}\K\|^2_2\leq\|MP_{\m+L^2}R_{\alpha,L}\bar
    P_{\m+\alpha^2}\|^2_2\|\K\|^2_{\rm op}\\\leq\int_{\m}^\infty
    dx\,\varphi(x)^2\int_{-\infty}^{\m+\alpha^2}dy\, \frac{1}{4\pi(\alpha+L)}e^{-4L^3/3+2\alpha^3/3-2Lx
      -2\alpha y-\frac{(x+y-2\m-\alpha^2)^2}{4(\alpha+L)}},
  \end{multline}
  where we have performed the change of variables $x\mapsto x+L^2$. The inner Gaussian
  integral gives
  \[\frac{1}{4\sqrt{\pi(\alpha+L)}}e^{-4L^3/3+2(\alpha-L)x-4\m\alpha+4L\alpha^2+8\alpha^3/3}
  \left[1+{\rm erf}\!\left(\tfrac12(\alpha+L)^{-1/2}[4\alpha(\alpha+L)+x-\m]\right)\right],\]
  where ${\rm erf}(z)=2\pi^{-1/2}\int_0^zdt\,e^{-t^2}\leq1$. Then
  \[\|P_{\m+L^2}R_{\alpha,L}\bar P_{\m+\alpha^2}\|^2_2\leq
  Ce^{-4L^3/3+4L\alpha^2}\int_{\m}^\infty dx\,\varphi(x)^2e^{2(\alpha-L)x},\] which clearly goes to 0
  as $L\to\infty$.

  On the other hand one can check that $e^{-(\alpha+L)H}$ has integral kernel given by
  \[e^{-(\alpha+L)H}(x,y)=\frac{1}{\sqrt{4\pi(\alpha+L)}}e^{-Lx-\alpha
    y+(\alpha^3+L^3)/3-\frac{(x-y)^2}{4(\alpha+L)}}\] (this is done either by applying the
  Feynman-Kac and Cameron-Martin-Girsanov formulas as in \cite{cqr}, or directly by
  integrating this kernel against the kernel of $e^{(\alpha+L)H}$, which is given,
  similarly to \eqref{eq:eLHK}, by $\int_{-\infty}^\infty
  d\lambda\,e^{-(\alpha+L)\lambda}\Ai(x-\lambda)\Ai(y-\lambda)$). Note that this kernel is
  the same as the one given in \eqref{eq:R} only without the reflection in the Gaussian
  term. It is easy to check then that the same calculation as the one in the above
  paragraph shows that $\|MP_{\m+L^2}e^{-(\alpha+L)H}\bar
  P_{\m+\alpha^2}\K\|_2\longrightarrow0$ as $L\to\infty$. This finishes the proof of
  \eqref{eq:omega}.
\end{proof}

\begin{proof}[Proof of Lemma \ref{lem:goodBCH}]
  By \eqref{eq:eLHK} and \eqref{eq:R}, the kernel of $e^{(\alpha+L)H}\K R_{\alpha,L}$ is given by
  \begin{multline}
    e^{(\alpha+L)H}\K R_{\alpha,L}(x,y)=\int_{-\infty}^0d\lambda\int_{-\infty}^\infty
    dz\,e^{(\alpha+L)\lambda}
    \Ai(x-\lambda)\Ai(z-\lambda)\\
    \cdot\frac1{\sqrt{4\pi(\alpha+L)}}e^{-(z+y-\alpha^2-2\m^2-L^2)^2/4(\alpha+L)+\alpha^3/3+L^3/3-Lz-\alpha
      y}.
  \end{multline}
  By completing the square in $z$ in the exponential, the $z$ integral can be seen as a
  heat kernel applied to an Airy function. Using the formula
  $e^{t\Delta}\Ai(x)=e^{2t^3/3+tx}\Ai(x+t^2)$ (see for instance Proposition
  \ref{a1-prop:mLaplacian} in \cite{quastelRemAiry1}) we obtain after some manipulations
  \[e^{(\alpha+L)H}\K
  R_{\alpha,L}(x,y)=\int_{-\infty}^0d\lambda\,e^{2\alpha^3+2\alpha\m^2-2\alpha y}
  \Ai(x-\lambda)\Ai(2\m^2+2\alpha^2-y-\lambda),\]
  which corresponds to the claimed formula.
\end{proof}

We turn finally to the proof of Proposition \ref{prop:tail}. It relies on the following simple but
very useful observation:

\begin{lem}\label{lem:detBd}
  If $A$ is a trace class operator on a Hilbert space $\ch$,
  \[\det(I+A)-1-\tr(A)\leq\tfrac12\|A\|_1^2e^{\|A\|_1}.\]
\end{lem}

\begin{proof}
  Let $\Lambda^n(A)=A\otimes\dots\otimes A$ ($n$ times, where $\otimes$ denotes the tensor product of
  operators), which is an operator in the Hilbert space $\Lambda^n(\ch)$ known as the
  alternating product, see \cite{simon} for more details. Then
  \[\det(I+A)=\sum_{k=0}^\infty\tr(\Lambda^k(A)),\]
  (this equality can be taken as the definition of the Fredholm determinant, see for instance (3.5) in
  \cite{simon}). Of course $\Lambda^0(A)=I$ and $\Lambda^1(A)=A$, so all we need to show
  is that
  \[\sum_{k=2}^\infty\tr(\Lambda^k(A))\leq\tfrac12\|A\|_1^2e^{\|A\|_1}.\]
  But this follows directly from the inequality $|\!\tr(\Lambda^k(A))|\leq\|\Lambda^k(A)\|_1\leq(k!)^{-1}\|A\|^k_1$ (see (3.1) and
  (3.4) in \cite{simon}).
\end{proof}

\begin{proof}[Proof of Proposition \ref{prop:tail}]
  Using the inequality (Theorem 3.4 in \cite{simon})
  \begin{equation}
    \label{eq:detCont}
    \left|\det(I+A)-\det(I+B)\right|\leq\|A-B\|_1e^{1+\|A\|_1+\|B\|_1}
  \end{equation}
  for any two trace class operators $A$ and $B$, we can write
  \begin{align}
    1-G^{2\to1}_\alpha(m)&=1-\det(I-P_mK^1_\alpha P_m-P_mK^2_\alpha P_m)\\
    &\leq1-\det(I-P_mK^1_\alpha P_m)+\|P_mK^2_\alpha P_m\|_1e^{1+2\|P_mK^1_\alpha P_m\|_1+\|P_mK^2_\alpha P_m\|_1}.
  \end{align}
  We will show below that
  \begin{equation}
    \|P_mK^1_\alpha P_m\|_1\leq c\,e^{-\frac23m^{3/2}},\label{eq:K1anorm}
  \end{equation}
  while a similar (and simpler) calculation gives the estimate $\|P_mK^2_\alpha P_m\|_1\leq
  ce^{-\frac43m^{3/2}}$ (this is in fact it is the same standard calculation based on
  \eqref{eq:detCont} that gives the estimate $1-F_{\rm GUE}(m)\leq
  c\,e^{-\frac43m^{3/2}}$). Therefore all we need to show is that
  \[1-\det(I-P_mK^1_\alpha P_m)\leq c\,e^{-\frac43m^{3/2}}\]
  and thus, in view of Lemma \ref{lem:detBd}, the proof it will be enough to show that
  \begin{equation}
    \label{eq:trBds}
    \tr(P_mK^1_\alpha P_m)\leq cm\,e^{-\frac43m^{3/2}+\alpha m}\qquad\text{and}\qquad
    \|P_mK^1_\alpha P_m\|_1\leq c\,e^{-\frac23m^{3/2}}.
  \end{equation}

  The trace can be computed directly: letting $\bar\alpha=\max\{0,\alpha\}$ we have
  \[\tr(P_mK^1_\alpha P_m)=\int_{m}^\infty dx\int_0^\infty
  d\lambda\,e^{\alpha\lambda}\Ai(x-\lambda+\bar\alpha)\Ai(x+\lambda+\bar\alpha).\] We
  split the integral according on the regions $\{\lambda<x\}$ and $\{\lambda\geq x\}$. On
  the first region, changing variables $\lambda\mapsto x\gamma$, and since the Airy function
  is decreasing and positive on the positive axis and we may assume $m>0$, the integral is
  bounded by
  \[\int_{m}^\infty dx\int_0^1d\gamma\,xe^{\alpha x\gamma}\Ai((1-\gamma)x)\Ai((1+\gamma)x)
  \leq c\int_{m}^\infty dx\int_0^1d\gamma\,xe^{\alpha
    x\gamma-\frac23[(1-\gamma)^{3/2}+(1+\gamma)^{3/2}]x^{3/2}},\] where we used the first
  of the estimates
  \begin{equation}
    |\!\Ai(x)|\leq c\,e^{-\frac23x^{3/2}}\quad\text{for $x>0$},\qquad
    |\!\Ai(x)|\leq c\quad\text{ for $x\leq0$}\label{eq:airybd}
  \end{equation}
  (see (10.4.59-60) in \cite{abrSteg}). The term in brackets in the above estimate is
  larger than 2, and thus the whole integral is bounded by $cme^{\alpha
    m-\frac43m^{3/2}}$. On the other region we have similarly, using \eqref{eq:airybd},
  that the integral is bounded by
  \[\int_{m}^\infty dx\int_1^\infty d\gamma\,xe^{\alpha
      x\gamma}\Ai((1-\gamma)x+\bar\alpha)\Ai((1+\gamma)x)
    \leq c\int_{m}^\infty dx\int_1^\infty d\gamma \,xe^{\alpha
      x\gamma-\frac23(1+\gamma)^{3/2}x^{3/2}},\]
  which again can be seen to be bounded by $cme^{\alpha m-\frac43m^{3/2}}$. This gives the
  first bound in \eqref{eq:trBds}.

  For the second one 
  %
  write $K_\alpha^1=(B^1e^{-\xi}P_0)(P_0e^{(1+\alpha)\xi}B^2)$, where
  $B^1(x,\lambda)=\Ai(x-\lambda+\bar\alpha)$ and
  $B^2(\lambda,y)=\Ai(y+\lambda+\bar\alpha)$. Then
  \begin{equation}
    \|P_mK_\alpha^1P_m\|_1\leq\|P_mB^1e^{-\xi}P_0\|_2\|P_0e^{(1+\alpha)\xi}B^2P_m\|_2.\label{eq:K1bd}
  \end{equation}
  The square of the first norm equals
  \[\int_m^\infty dx\int_{0}^\infty
  d\lambda\Ai(x-\lambda+\bar\alpha)^2e^{-2\lambda}\leq\int_m^\infty dx\,e^{-2(x+\bar\alpha)}\int_{-\infty}^\infty
  d\lambda\Ai(-\lambda)^2e^{-2\lambda}\leq c\,e^{-2m},\] thanks to
  \eqref{eq:airybd}. Similarly, the square of the second norm equals
  \[\int_m^\infty dy\int_{0}^\infty
  d\lambda\Ai(y+\lambda+\bar\alpha)^2e^{2(1+\alpha)\lambda} \leq c\int_m^\infty
  dy\int_{0}^\infty d\lambda\,e^{-\frac43(y+\lambda+\bar\alpha)^{3/2}+2(1+\alpha)\lambda}\leq
  c\,e^{-\frac43m^{3/2}},\] again thanks to \eqref{eq:airybd}. Using these two bounds
  in \eqref{eq:K1bd} gives the second estimate in \eqref{eq:trBds} and finishes
  the proof.
  \end{proof}

  Observe that for $\alpha<-1$ the last bound in the proof above can be upgraded to
  $ce^{2(1+\alpha)-\frac43m^{3/2}}$, giving
  \begin{equation}
    \label{eq:G21toK}
    \|P_mK_\alpha^1P_m\|_1\leq c\,e^{\alpha-m-\frac23m^{3/2}}\xrightarrow[\alpha\to-\infty]{}0,
  \end{equation}
  which we used in justifying \eqref{eq:toGUE}. A similar (although slightly more
  complicated) estimate gives $\|P_m\overline K^1_\alpha
  P_m\|_1\xrightarrow[\alpha\to\infty]{}0$, which was used in the derivation of \eqref{eq:toGOE}.

\printbibliography[heading=apa]

\end{document}